\documentclass[11pt,reqno]{amsart}
\usepackage{bbm}
\usepackage{cite}
\usepackage{a4wide}

\usepackage{amsmath,amssymb} 

\usepackage{color}
\usepackage{tkz-graph}
\tikzset{
  LabelStyle/.style = { rectangle, rounded corners, draw,
                      minimum width = 1em, 
                       font =  },
  VertexStyle/.append style = { inner sep=3pt,
                               font = \large},
  EdgeStyle/.append style = {->, bend left} }

\newenvironment{psmallmatrix}
  {\left(\begin{smallmatrix}}
  {\end{smallmatrix}\right)}

\theoremstyle{plain}
\newtheorem{theorem}{Theorem}
\newtheorem{prop}[theorem]{Proposition}

\newtheorem{coro}[theorem]{Corollary}

\theoremstyle{definition}

\newtheorem{remark}[theorem]{Remark}

\newtheorem{example}[theorem]{Example}

\newcommand{\Z}{{\mathbb Z}}

\newcommand{\R}{{\mathbb R}}

\newcommand{\C}{{\mathbb C}}

\newcommand{\exend}{\hfill $\Diamond$}

\begin{document}

\title[Binary substitutions of constant length]{
Lyapunov Exponents for Binary  Substitutions \\
of Constant Length
}

\author{Neil Ma\~nibo}
\date{\today}
\address{Fakult\"at f\"ur Mathematik, Universit\"at Bielefeld, \newline
\hspace*{\parindent}Postfach 100131, 33501 Bielefeld, Germany}
\email{cmanibo@math.uni-bielefeld.de }

\begin{abstract}  

A method of confirming the absence of absolutely continuous diffraction for binary constant-length substitutions, which are primitive and aperiodic, via the positivity of Lyapunov exponents derived from the corresponding Fourier matrices is presented, which provides an approach that is independent of previous results on the basis of Dekking's criterion.  
\end{abstract}



\maketitle

\section{Introduction}

Classifying aperiodic substitutions via their diffraction  profile remains to be one of the key problems in the spectral theory of substitution systems.
Various works have already clarified the correspondence between the dynamical (which is in general richer) and the diffraction spectra of such systems, with Solomyak, Lee, and Moody proving in \cite{LMS} that the notion of having a pure point dynamical spectrum and diffraction spectrum is equivalent for Delone sets with finite local complexity and uniform cluster frequencies. Baake, Lenz, and van Enter later showed in \cite{BLvE} that a recovery of the dynamical spectrum for systems with mixed spectrum is possible via the diffraction of the substitution's factors. These inclusion results are crucial as one seeks for generalizable criteria satisfied by substitutions bearing the same spectral type.
 
A lot of progress has been made in the constant length case; i.e., when the substituted words all have the same length. We mention some of them, so as to see that the results we obtain here are consistent to what have previously been shown. 
Dekking has shown in \cite{Dekking} that a constant length substitution with height 1 has pure point dynamical spectrum if and only if it has a coincidence, and has partly continuous spectrum otherwise.
Queff\'{e}lec, in her extensive survey in \cite{Queffelec}, paid particular attention to the bijective Abelian cases, which were shown to have Riesz products converging weakly to singular measures that generate the maximal spectral type. Bartlett explicitly proved in \cite{Bartlett} that these systems have singular dynamical spectra.  

The main contribution of this paper is the detection of the lack of absolutely continuous diffraction, which has proven to be quite elusive in the one--dimensional case. Frank has detailed a construction via Hadamard matrices of $d-$dimensional substitutions  which have a Lebesgue component of even multiplicity in their spectra in \cite{Frank}.  Chan and Grimm were also able to show in \cite{ACspec} that a family of four-letter substitutions in one dimension can be obtained via sequences satisfying the root $N$ property, which by Queff\'{e}lec's arguments in \cite{Queffelec} automatically have an absolutely continuous diffraction component. This is an extension of the famous Rudin--Shapiro substitution, which has long been known to satisfy that condition. Recently, Berlinkov and Solomyak have shown in \cite{SSCL} that an absolutely continuous component requires that the substitution matrix $M_{\varrho}$ of such $q-$letter substitution  has an eigenvalue whose modulus is $\sqrt{q}$.
An example beyond the constant length regime is yet to be found.

This paper is organized as follows: in Section \ref{subs}, we give a brief review of substitutions and their geometric realization as inflation rules. This geometric description  is encoded in the Fourier matrices and the algebra they generate, which will be discussed in Section \ref{ida}. We introduce an alternative to the usual calculation of the diffraction measure, and we explain how the Fourier matrices figure in this renormalisation scheme which can be restricted to the absolutely continuous component. 
The fifth section describes the iteration in the preceding section in the language of cocycles via Lyapunov exponents and why their positivity rules out the existence of the component in question. In the last section, we show that for all considered cases the exponents are indeed positive and we present some examples. 
We end with some concluding remarks and mention the outlook of extending this method to more general cases. 
 
\section{Substitutions}\label{subs}
  We consider a constant length substitution $\varrho$
on a binary alphabet $\mathcal{A}_{2}=\left\{ a,b\right\} $ given by 
$$\varrho:\mathcal{A}_{2}\rightarrow\mathcal{A}_{2}^{+},\varrho:\begin{cases}
a\mapsto w_{a}\\
b\mapsto w_{b}
\end{cases}$$
where $\mathcal{A}_{2}^{+}$ is the set of all finite words with letters in $\mathcal{A}_{2}$ and  $\left|w_{a}\right|=\left|w_{b}\right|=L\geq2$, where $\left|w\right|$ is the \emph{length} of the word $w$. We define  $\mathcal{A}^{*}_{2}=\mathcal{A}^{+}_{2}\cup \left\{\epsilon\right\}$ where $\epsilon$ is the empty word. Note that $\varrho$ extends to $\mathcal{A}_{2}^{*}$ by concatenation. Analogously, the same holds for the sets $\mathcal{A}^{\mathbb{N}_0}_2$ and $\mathcal{A}^{\mathbb{Z}}_2$, which are the sets of infinite and bi-infinite words with letters in $\mathcal{A}_2$, respectively. 
Throughout this paper, we will be dealing with such substitutions which are primitive and aperiodic.  We refer the reader to  \cite[Ch.~4]{BG:AO} for a more detailed exposition.

We choose any bi-infinite fixed point $w\in \mathcal{A}_{2}^{\mathbb{Z}}$ of (possibly a power of) $\varrho$
arising from a legal seed $a_{i}\left|a_{j}\right.$ (which is possible due to the primitivity of $\varrho$). The \emph{symbolic hull} $\mathbb{X}$ defined by $w$ is the orbit closure of the set of all the translates of $w$ under the shift operator $S$ defined by $\big(Sw\big)_{j}:=w_{j+1}$.

By viewing $\varrho$ as
an inflation rule on two tiles of the same length,
we make sense of tile positions. We label each interval by its left endpoint, and so its position is the location of this identifying point. The substitution's geometric and symbolic pictures coincide
assuming we associate to each letter $\alpha$ an interval $I_{\alpha}$ of unit length. 

The coloured point set $\varLambda^{w}$ associated with $w$ can be identified with its embedding in $\Z$, which decomposes into
$$\varLambda^{w}=\varLambda_{a}\overset{\cdot}{\cup}\varLambda_{b},$$
where $\varLambda_{a}$ and $\varLambda_{b}$ precisely are the collection
of control points associated to $a\text{ and }b$ . As with the symbolic one, we build the \emph{geometric hull} by taking the closure of the set of $\Z-$translates of $\varLambda$, i.e., 
$$\mathbb{Y}=\overline{\left\{t+\varLambda^{w}\mid t\in \Z\right\}}.$$
Here, $\mathbb{X}$ and $\mathbb{Y}$ are mutually locally derivable (MLD) and define topologically conjugate dynamical systems; see \cite{BG:AO}.
The \emph{diffraction measure} $\widehat{\gamma}^{}_{\omega}$ associated to the point set $\varLambda^{w}$ is derived from the \emph{autocorrelation} $\gamma^{}_{\omega}$, where $\omega$  is a Dirac comb describing $\varLambda^{w}$. The diffraction is the same for every element in $\mathbb{Y}$ when the system is uniquely ergodic \cite{BL}, which will always be the case due to our assumptions on $\varrho$.
We do not discuss the actual construction here and instead refer the reader to\cite{RPMD} for a brief run-through and \cite[Ch.~9]{BG:AO} for a comprehensive discussion.

 \section{Fourier Matrices and IDA}\label{ida}

We define the Fourier matrix $B\left(k\right)$ associated to $\varrho$ whose entries are given by  
$$
B_{\alpha\beta}\left(k\right)=\sum_{x\in T_{\alpha\beta}}e^{2\pi i kx},
$$
where $T_{\alpha\beta}$ is the set of all relative positions of tiles of type
$\alpha$ in the patch $\varrho\left(\beta\right)=w_{\beta}$, where $\alpha,\beta \in \left\{a,b\right\}$. 
We define the total set $S_T$ to be the set of all relative positions of prototiles in level--1 supertiles corresponding to $w_a$ and $w_b$, i.e. $S_T=\bigcup_{\alpha,\beta}T_{\alpha\beta}$. 
It is clear from the geometry of our inflation tiling that $S_T=\left\{0,1,\ldots,L-1\right\}$, corresponding bijectively to the columns of our substitution. 
With this, we can express $B\left(k\right)$ as
$$B\left(k\right)=\sum_{x\in S_T}e^{2\pi i kx}D_x$$
where $D_x$ is the 0-1 matrix that describes the substitution rule on the column at position $x$. These matrices are also referred to as \emph{instruction matrices} in \cite{Bartlett} and \emph{digit matrices} in older papers and is defined as 
$$\left(D_x\right)_{\alpha\beta}=\begin{cases} 1, & \text{ if } \left(w_{\beta}\right)_x=\alpha, \\
0, & \text{ otherwise}.
\end{cases}$$
Note that $M_{\varrho}=B(0)$.
Next, we consider the complex algebra $\mathcal{B}$ generated by the family $\left\{B\left(k\right)\mid k\in \R\right\}$. We remark that $\mathcal{B}$ and the algebra $\mathcal{B}_D$  generated by the digit matrices $D_z$ coincide, which we refer to as the \emph{inflation displacement algebra (IDA)} of $\varrho$; compare with \cite{BG:Pair}. 

\begin{example}
We consider the Thue--Morse substitution given by 
$\varrho^{}_{TM}: a\mapsto ab, b\mapsto ba$. Its Fourier matrix  reads $$B\left(k\right)=\begin{pmatrix}
1 & e^{2\pi ik}\\
e^{2\pi ik} & 1
\end{pmatrix},$$ with the digit matrices being $D_0=\begin{psmallmatrix}
1 & 0\\
0 & 1
\end{psmallmatrix}, D_1=\begin{psmallmatrix}
0 & 1\\
1 & 0
\end{psmallmatrix}$. 
On the other hand,  the period doubling substitution, which reads $\varrho_{\text{pd}}:a\mapsto ab, b\mapsto aa$, has the Fourier matrix
$$B\left(k\right)=\begin{pmatrix}
1 & 1+e^{2\pi ik}\\
e^{2\pi ik} & 0
\end{pmatrix},$$ with the digit matrices  $D_0=\begin{psmallmatrix}
1 & 1\\
0 & 0
\end{psmallmatrix}, D_1=\begin{psmallmatrix}
0 & 1\\
1 & 0
\end{psmallmatrix}$. 
Their IDAs are given by 
$$\mathcal{B}_{TM}=\left\{\begin{psmallmatrix}
c_1 & c_2\\
c_2& c_1
\end{psmallmatrix} \mid c_1,c_2 \in \C\right\}, \quad \mathcal{B}_{\text{pd}}=\left\{\begin{psmallmatrix}
 c_1+c_2 & c_1\\
c_3 & c_2+c_3
\end{psmallmatrix} \mid c_1,c_2,c_3 \in \C\right\}, $$
which are $2-$dimensional and $3-$dimensional, respectively. 
\exend 
\end{example}

A binary substitution is said to have a \emph{coincidence} at position $x$, if at that specific position $\left(w_a\right)_x=\left(w_b\right)_x$. 
Substitutions whose columns are all permutations of $\left\{a,b\right\}$ are called \emph{bijective}. 
It is easy to see that a binary substitution either is bijective or has a coincidence.  
The following proposition provides a classification of algebras arising from binary substitutions. 

\begin{prop}
Given any binary constant length substitution $\varrho$, its corresponding IDA is either isomorphic to that of Thue-Morse or period doubling; i.e., 
$\mathcal{B}\cong \mathcal{B}_{\text{TM}}$ or $\mathcal{B}\cong \mathcal{B}_{\text{\emph{pd}}}$, where the former holds whenever $\varrho$ is bijective. 
\end{prop}

We recall that a complex $2\times2$ matrix algebra is only \emph{irreducible} if it is the full algebra $\emph{M}_{2}\left(\C\right)$, which is of dimension 4. 
Since the possible IDAs are either of dimension 2 or 3, each one of them is reducible and bears at least one invariant subspace.

\section{Renormalisation}

We extend the method of deriving \emph{exact renormalisation relations} for pair correlation functions (which in theory is possible for any substitution that is locally recognizable) done for particular cases in \cite{BG:Pair} to arbitrary binary constant length substitutions. In this scheme, these relations are used to decompose the \emph{autocorrelation} $\gamma^{}_{\Lambda}$ and lift the renormalisation to the components of the \emph{diffraction measure} $\widehat{\gamma^{}_{\Lambda}}$ via the Fourier matrices, which we briefly describe below.  

Consider the relative frequency $\nu_{\alpha\beta}\left(z\right)$ of distance
$z$ from a left endpoint of a tile of type $\alpha$ to a tile of type
$\beta$, which is given directly as 
$$\nu_{\alpha \beta}\left(z\right)=\underset{r\rightarrow \infty}{\lim}\dfrac{\textrm{card}\left(\Lambda_{\alpha,r}\cap\left(\Lambda_{\beta,r}-z \right)\right)}{\textrm{card}\left(\Lambda_r \right)},$$
where $\Lambda_{\alpha,r}=\Lambda_{\alpha}\cap\left[-r,r\right]$. These frequencies exist due to the strict ergodicity of the dynamical system $\big(\mathbb{X},S\big)$, which follows from the primitivity of $\varrho$.
 From these pair correlation functions, four positive pure point measures can be built, namely 
$$
\boldsymbol{\nu}_{\alpha\beta}=\sum_{z\in\Lambda_{\beta}-\Lambda_{\alpha}}\nu_{\alpha\beta}\left(z\right)\delta_{z}.
$$

 The column measure vector given by $\boldsymbol{\nu}=\left(\boldsymbol{\nu}_{aa},\boldsymbol{\nu}_{ab},\boldsymbol{\nu}_{ba},\boldsymbol{\nu}_{bb}\right)^{\mathrm{T}}$ is Fourier transformable and $\widehat{\boldsymbol{\nu}}$ 
satisfies 
$$
\widehat{\boldsymbol{\nu}}=\dfrac{1}{\lambda^{2}}A\left(.\right)\left(f^{-1}\cdot\widehat{\boldsymbol{\nu}}\right),
$$
where $f\left(x\right)=\lambda x$, and $A\left(k\right)=B\left(k\right)\otimes\overline{B\left(k\right)}$, $\lambda$ being the inflation multiplier of $\varrho$. 
Furthermore, each spectral component $\left(\widehat{\boldsymbol{\nu}}\right)_{\texttt{t}}, \texttt{t}\in\left\{\texttt{pp},\texttt{ac},\texttt{sc}\right\}$ 
individually satisfies
the same renormalisation equation. In particular, if one constructs the density vector $h$,  whose components are the Radon--Nikodym densities $h_{\alpha\beta}$  associated to $\left(\widehat{\boldsymbol{\nu}}_{\alpha\beta}\right)_{\texttt{ac}}$,
 it must be true that
$$
h\left(\dfrac{k}{\lambda}\right)=\dfrac{1}{\lambda}A\left(\dfrac{k}{\lambda}\right)h(k)
$$
for a.e. $k\in\mathbb{R}$. The dimension
of the iterative equation can further be reduced via the fact that $H=\left(h_{\alpha\beta}\left(k\right)\right)$
is a Hermitian positive definite matrix of rank at most 1, from which
we obtain an inward iteration given by 
$$
v\left(\frac{k}{\lambda^{n}}\right)=\lambda^{-n/2}B\left(\frac{k}{\lambda^{n}}\right)\cdot\ldots\cdot B\left(\frac{k}{\lambda}\right)v(k)
$$
where $H=\left|v\left\rangle \right\langle v\right|,v\left(k\right)=\left(v^{ }_{0}(k),v^{ }_{1}(k)\right)^{\mathrm{T}}$, which takes the form of an almost projector, i.e. $H^2=\|v\|^2\cdot H$, where $v$ is not necessarily of norm 1 and can even be the zero vector. Element-wise, one has $H_{ij}=v^{}_{i}\overline{v^{ }_{j}}$.  This dimension reduction step is not straightforward and is even more subtle for substitutions on more than two letters. We refer the reader to \cite[Lem.~4.3]{Family} for the proof of this in the binary case. 
Likewise, we also have the outward iteration, which reads
\begin{equation}\label{iteration}
v(\lambda^{n}k)=\lambda^{n/2}B^{-1}(\lambda^{n-1}k)\cdot\ldots\cdot B^{-1}(k)v(k),
\end{equation}
which is well-defined for a.e. $k$ when the determinant of $B^{-1}(k)$ does not vanish everywhere.
The rest of our discussion will focus on the outward iteration. 
\begin{remark}
The analysis is anchored on the growth of the norm
of $v$ under this iteration, i.e., since $\left| v(\lambda^{n}k)\right| \rightarrow \infty$ as $n\rightarrow\infty$
would be in contradiction to the
translation boundedness of $\left(\widehat{\boldsymbol{\nu}}\right)_{\text{ac}}$, we must have
  $h=0$ in the Lebesgue sense and hence $\left(\widehat{\boldsymbol{\nu}}\right)_{\text{ac}}=0.$ This blow-up is related to the Lyapunov exponents derived from the matrix product $B^{-1}(\lambda^{n-1}k)\cdot\ldots\cdot B^{-1}(k)$ which will be discussed in the next section.
  \end{remark}

\begin{remark}\label{nonpisot}
We note that $\lambda$ is used in this section to remind that no restriction on the inflation factors is enforced for this renormalization procedure to be applicable. A case where $\lambda$ is non-Pisot is tackled rigorously in \cite{WIP1}, and entire family to which the former belongs is dealt with in \cite{Family}. From here onwards, we use the notation $\lambda =L$, where $L \in \mathbb{N}$ is  just the common length  of each substituted word, since we are considering substitutions of constant length.
\end{remark}

\section{Lyapunov Exponents}

To the iteration in Eq. (\ref{iteration}), we associate a  matrix cocyle given by $$B^{\left(n\right)}(k)=B^{-1}(L^{n-1}k)\cdot\ldots\cdot B^{-1}(k).$$
The matrix $B^{-1}(k)$ exists whenever
$\det B(k)\neq0$. This existence holds for a.e. $k$ as long as $\det B(k)\not\equiv 0.$
Since the map $T:k\mapsto\left(Lk\right)\text{mod}1$ is ergodic on the interval
$\left[0,1\right)$ when $L\in\mathbb{N}$, Oseledec's multiplicative
ergodic theorem guarantees Lyapunov regularity a.e. in our case; i.e., the existence
of the Lyapunov exponents for a.e. $k,$ which are given by
$$
\chi^{}_{i}(v)=\lim_{n\rightarrow\infty}\dfrac{1}{n}\log\|L^{n/2}B^{\left(n\right)}(k)v(k)\|
$$
and of a filtration $\mathcal{V}:=\left\{ 0\right\} =\mathcal{V}_{0}\subsetneq\mathcal{V}_{1}\subsetneq\mathcal{V}_{2}=\mathbb{C}^{2}$
such that $\chi^{}_{1}$ is the corresponding exponent for all $0\neq v\in\mathcal{V}_{1}$ and
$\chi^{}_{2}$ for all $v\in\mathcal{V}_{2}\setminus\mathcal{V}_{1}$. The exponents, if they exist, are independent of the matrix norm $\left\|\cdot\right\|$ used. For general background on this topic, we refer the reader to \cite{BP:NH,Viana}. 

There are at most two distinct exponents for the binary case, 
which by regularity satisfy 
$$
\chi^{}_\mathrm{min}+\chi^{}_\mathrm{max}=\lim_{n\rightarrow\infty}\dfrac{1}{n}\log\left|\det B^{\left(n\right)}(k)\right|+\log L,
$$
where $\chi^{}_\mathrm{min},\chi^{}_\mathrm{max}$ are given by 
\begin{align}\label{exp}
\chi^{}_\mathrm{min}(k) & = \log\sqrt{L}-\underset{n\rightarrow\infty}{\lim}\dfrac{1}{n}\log\left\Vert \big(B^{\left(n\right)}(k)\big)^{-1}\right\Vert \nonumber \\ 
\chi^{}_\mathrm{max}(k) & = \log\sqrt{L}+\underset{n\rightarrow\infty}{\lim}\dfrac{1}{n}\log\left\Vert B^{\left(n\right)}(k)\right\Vert. 
\end{align}

\begin{remark}
The positive measure $\left(\widehat{\boldsymbol{\nu}}^{}_{aa}+\widehat{\boldsymbol{\nu}}^{}_{bb}\right)_{\texttt{ac}}$ could be represented by the Radon--Nikodym density $h^{\prime}(k)=\left|v_0(k)\right|^2+\left|v_1(k)\right|^2$, where $v_0,v_1$ are the components of the vector $v(k)$ being iterated via the cocycle $B^{\left(n\right)}(k)$. Suppose that the minimum exponent $\chi^{}_\mathrm{min}(k)=D>0$. Then we have that $\exists C$ such that 
$h^{\prime}(L^{n}k)\geq Ce^{Dn}h^{\prime}(k)$; compare \cite{BP:NH}, which contradicts the translation boundedness of $\left(\widehat{\boldsymbol{\nu}}^{}_{aa}+\widehat{\boldsymbol{\nu}}^{}_{bb}\right)_{\texttt{ac}}$, and hence effectively of $\widehat{\boldsymbol{\nu}}$. The only way to escape this situation is $v(k)=0$ and hence $h(k)\equiv 0$ for Lebesgue a.e. $k$.  
It suffices to show that $\chi^{}_\mathrm{min}$ is positive to show that $\left(\widehat{\boldsymbol{\nu}}\right)_{\text{ac}}=0$, since the contradiction follows automatically for the subspace of $\C^2$ corresponding to the larger exponent. 

\end{remark}
\begin{remark}
It is worth noting that for the inward iteration,
with the associated cocycle being $B^{\left(n\right)}\left(k\right)=B\left(\frac{k}{\lambda^{n}}\right)\cdot\ldots\cdot B\left(\frac{k}{\lambda}\right)$, the Lyapunov exponents can easily be computed using the eigenvalues of the substitution matrices, and the subspaces $\mathcal{V}_{1},\mathcal{V}_{2}$ are spanned by the corresponding eigenvectors. This is because $B\left(\frac{k}{\lambda^{n}}\right)\rightarrow M_{\varrho}$ as $n\rightarrow \infty$, and hence the cocyle behaves more like a power $M^{n}_{\varrho}$ of the substitution matrix. However, in this case, it is possible to get zero and negative infinity as limits, which do not contradict the properties of $\left(\widehat{\boldsymbol{\nu}}\right)_{\text{ac }}$. This means the examination of the other iteration is needed.

\end{remark}

The next section is devoted to showing the positivity sought, which we achieve via bounds on norms of polynomials satisfying some specific properties.

\section{Positivity of $\chi_{\min}$}

In general, the explicit computation of the Lyapunov exponents for analytic cocyles is a hard problem, but the existence of a common invariant subspace helps in this case.
Since there are only four possible column types in the binary case, it is possible to express the entries of the Fourier matrices in terms of trigonometric polynomials associated to these column types. 	
First, we construct the  sets 
\begin{eqnarray*}
C_{a} & = & \left\{ i\mid\left(w^{}_{a}\right)_{i}=\left(w^{}_{b}\right)_{i}=a\right\} \\
C_{b} & = & \left\{ i\mid\left(w^{}_{a}\right)_{i}=\left(w^{}_{b}\right)_{i}=b\right\} \\
P_{a} & = & \left\{ i\mid\left(w^{}_{a}\right)_{i}=a,\left(w^{}_{b}\right)_{i}=b\right\} \\
P_{b} & = & \left\{ i\mid\left(w^{}_{a}\right)_{i}=b,\left(w^{}_{b}\right)_{i}=a\right\} 
\end{eqnarray*}
where $C_{a},C_{b}$ are the positions corresponding to coincidences
and $P_{a},P_{b}$ are the bijective positions. These sets are disjoint
and their union is $S_T=\left\{ 0,1,\ldots,L-1\right\} $. From these we define polynomials
\begin{align*}
P_{1}&= \sum_{z\in C_{a}}u^{z} &P_{2} &=\sum_{z\in C_{b}}u^{z}\\
Q&= \sum_{z\in P_{a}}u^{z} &R &=\sum_{z\in P_{b}}u^{z}
\end{align*}
where $P_{1}+P_{2}+Q+R=p^{}_{L}=1+u+\ldots+u^{L-1}$, $u=e^{2\pi ik}$. 

The Fourier matrix of $\varrho$ can be constructed from these
polynomials to be
$$
B\left(k\right)=\begin{pmatrix}P_{1}+Q & P_{1}+R\\
P_{2}+R & P_{2}+Q
\end{pmatrix},$$
with $\det B\left(k\right)=p^{}_{L}\cdot\left(Q-R\right).$
From this, we can compute from  Eq. (\ref{exp}) the sum of the exponents to be
\begin{equation}\label{sumexpo}
\log L-\dfrac{1}{n}\underset{m=0}{\overset{n-1}{\sum}}\log \left| \det  B\left(L^{m}k\right)\right|
\underset{a.e.\text{ }k}{\xrightarrow{n\rightarrow\infty}}\log L
-\underbrace{\int_{0}^{1}\log\left|p^{}_{L}\right|dk}_{=0}-\underbrace{\int_{0}^{1}\log\left|Q-R\right|dk}_{=m\left(Q-R\right)},
\end{equation}
where $m\left(Q-R\right)$ is the log--Mahler measure of the polynomial
$Q-R$. This convergence is not trivial, as one might expect, as $\log \left| \det B\left(L^{m}k\right)\right|$, which is almost periodic, might have singularities (at most countably many, when $\det B\left(k\right)\not\equiv 0$). It is obvious that the requirements of Weyl's criterion are not satisfied and that one must resort to some discrepancy analysis to show that its time average does have the same integral as its limit. We 
refer the reader to \cite{Average} for a more detailed account on this. 

Actual values of $m\left(Q-R\right)$ for specific polynomials can
be obtained via Jensen's formula, see \cite{Ahlfors}, which states 
\begin{equation}\label{Jensen}
m\left(f\right)=\int_{0}^{1}\log\left|f\left(e^{2\pi ik}\right)\right|dk=\log\left|a_{0}\right|+\sum_{i=1}^{n}\log\left(\max\left\{ 1,\left|\xi_{i}\right|\right\} \right),
\end{equation}
where $f=a_{0}+a_{1}u+\ldots+a_{n}u^{n}$, and $\xi_{i}$ are the roots
of $f$. From this, one can easily see that $m\left(p_L\right)=0$ because $p_L\mid u^{L}-1.$
\begin{remark}
From the disjointness of $P_a$ and $P_b$, $Q-R=0$
implies that $Q=R=0$, and hence the substitution is composed only of coincident columns. We then get exactly the same substituted word for each letter, i.e. $w_a=w_b$
and the periodicity of the hull $\mathbb{X}$ is immediate. For such cases, it is known \emph{a priori} that the diffraction spectrum is pure point. 
However, our method is not applicable because $\det B\left(k\right)=0$ for all $k$. 
 \end{remark}
 
 We now state the main result via the next theorem and the succeeding corollary. 
\begin{theorem}\label{ThmLyapunov}

The pointwise Lyapunov exponents of an aperiodic binary constant length substitution $\varrho$, for a.e. $k\in\R$, are given by 
\begin{align*}
\chi^{}_{\mathrm{max}}\left(k\right)&=\log\sqrt{L} \\ 
\chi^{}_{\mathrm{min}}\left(k\right)&=\log\sqrt {L}-m\left(Q-R\right).
\end{align*}
\end{theorem}

\begin{proof}

We can rewrite $\chi_{\mathrm{max}}$ in Eq.~\eqref{exp} as
$$\log\sqrt{L}-\underset{n\rightarrow\infty}{\lim}\dfrac{1}{n}\log\left\Vert B^{\text{ad}}(L^{n-1}k)\cdot\ldots\cdot B^{\text{ad}}(k)\right\Vert -\dfrac{1}{n}\underset{m=0}{\overset{n-1}{\sum}}\log \left| \det\left( B\left(L^{m}k\right)\right)\right|, $$
where the last term converges as $n\rightarrow \infty$ for a.e. $k$ to $m\left(Q-R\right)$ as with Eq.~\eqref{sumexpo}. 
 The extremal exponents then read 
\begin{align*}
\chi^{}_\mathrm{min}(k) & = \log \sqrt{L}-\underset{n\rightarrow\infty}{\lim}\dfrac{1}{n}\log\left\Vert B(k)\cdot\ldots\cdot B(L^{n-1}k)\right\Vert \\
\chi^{}_\mathrm{max}(k) & = \log L- \chi^{}_\mathrm{min}(k)-m\left(Q-R\right),
\end{align*}
where we have used that for $2\times2$ matrices, $\left\Vert A\right\Vert_{F}=\left\Vert A^{\text{ad}}\right\Vert_{F}$, with $\left\Vert\cdot\right\Vert_{F}$ being the Frobenius norm. 

We carry out the calculation of the minimum exponent by noting that $\C\left(\left(1,-1\right)^{\mathrm{T}}\right)$ is always an invariant subspace for any $B(k)$, i.e. 
$$B(k)\begin{pmatrix}1\\
-1\end{pmatrix}=\left(Q-R\right)\begin{pmatrix}1\\
-1\end{pmatrix}.$$ 
Comparing $\chi^{}_\mathrm{min}$ with the value we have for the iteration on this subspace, we get, for a.e. $k$,
$$\chi^{}_\mathrm{min}\left(k\right)\leq\log\sqrt{L}-\underset{n\rightarrow\infty}{\lim}\dfrac{1}{n}\underset{m=0}{\overset{n-1}{\sum}}\log \left| Q-R\right|\left(u^{L^m}\right)=\log\sqrt{L}-m\left(Q-R\right)\leq \log\sqrt{L}, $$
noting that $\log\sqrt{L}-m\left(Q-R\right)$ is a value for one of the exponents. 
From the previous inequality, this cannot be $\chi^{}_\mathrm{max}$ when $m\left(Q-R\right)>0$
since $\chi^{}_\mathrm{max}\left(k\right)\geq \log\sqrt{L} $.
Hence, in general, it must be the smaller of the two, and from their sum via Eq.~\eqref{sumexpo} we obtain the result. Equality of the two exponents holds when $m\left(Q-R\right)=0$, which due to a result by Kronecker happens exactly whenever $Q-R$ is a product of a monomial and a cyclotomic polynomial \cite{Kronecker}. 
  \end{proof}
\begin{coro}
Let $\varrho$  be as in the previous theorem. Then, the Lyapunov exponents $\chi^{}_\mathrm{min}\left(k\right)$ and $\chi^{}_\mathrm{max}\left(k\right)$
are positive for a.e. $k\in \R$.  
\end{coro}

\begin{proof}
It suffices to show that $m\left(Q-R\right)<\log \sqrt {L}.$
Consider $M\left(Q-R\right)=\text{exp}\left(m\left(Q-R\right)\right).$
By Jensen's inequality and the usual ordering of polynomial norms, we know that 
\begin{equation}\label{JensenIneq}
M\left(Q-R\right) < \left\Vert Q-R\right\Vert _{1} \leq \left\Vert Q-R\right\Vert _{2}
\end{equation}
where the strictness of the first inequality follows from the strict convexity of  $g(x)=e^{x}$ \cite{LL:Analysis}, and the second one is an equality only when $Q-R$ is a monomial. Moreover, we have  from Parseval's identity that
$$\left\Vert Q-R\right\Vert _{2}=\sqrt{L-\text{card}\left(C_{a}\cup C_{b}\right)}.$$

When the substitution has a coincidence, we have that $\text{card}\left(C_{a}\cup C_{b}\right)\geq1$. If $Q-R$ is a monomial, $\text{card}\left(C_{a}\cup C_{b}\right)=L-1$ and hence from Eq.~\eqref{Jensen} $$m\left(Q-R\right)=0.$$ 
Otherwise, Eq.~\eqref{JensenIneq} implies $m\left(Q-R\right)<\log\sqrt{L-\text{card}\left(C_{a}\cup C_{b}\right)}<\log\sqrt{L}$,
and therefore $\chi^{}_\mathrm{min}$ is always positive.

For bijective substitutions, we have $P_{1}=P_{2}=0$, and hence $Q-R$ is a Littlewood
polynomial of degree $L-1$; see \cite{Borwein}.
Since $Q-R$ is not a monomial, as $L\geq 2$ and $a_0=\pm 1$, Eq.~\eqref{JensenIneq} reads
$$
M\left(Q-R\right)<\left\Vert Q-R\right\Vert _{1}<\left\Vert Q-R\right\Vert _{2}=\sqrt{L}$$
and hence $\chi_\mathrm{min}>0$. 
Better bounds for $\left\|f\right\|_{1}$ actually exist in literature (see \cite{Borwein,Habsieger}), but are not needed for this positivity result. 
\end{proof}

\begin{example}
For $\varrho^{}_{TM}$ and $\varrho_{\text{pd}}$, we have that $(Q-R)^{}_{TM}=1-e^{2\pi ik},(Q-R)_{\text{pd}}=-e^{2\pi ik}$, and hence $m\left(Q-R\right)=0$ for both cases, implying that $\chi^{}_\mathrm{min}=\chi^{}_\mathrm{max}=\log\sqrt{2}$.
\exend
\end{example}
\begin{example} We note that the Lyapunov spectrum is not an IDA invariant. 
Consider the substitution $\varrho: a\mapsto abbab,b\mapsto baaba$, with corresponding Fourier matrix $$B(k)=\begin{pmatrix}
e^{6\pi ik} & e^{2\pi ik}+e^{4\pi ik}+e^{8\pi ik}\\
e^{2\pi ik}+e^{4\pi ik}+e^{8\pi ik} &  e^{6\pi ik}
\end{pmatrix}.
$$ 
Here, $\left(Q-R\right)\left(u\right)=u\left(-1-u+u^2-u^3\right)$, where $u=e^{2\pi ik}$. Note that its IDA is isomorphic to that of the Thue--Morse since it is bijective. 
One can explicitly compute that the Lyapunov exponents for the outward iteration for this substitution are 
\begin{align*}
\chi^{}_\mathrm{max}(k)&=\log\sqrt{5}\approx 0.8047 \\ 
\chi^{}_\mathrm{min}(k)&=\log\sqrt {5}-2\log\left(1.3562\right)\approx 0.1953.
\end{align*} 
 \exend
\end{example}
\begin{example}
As shown above, we do not have examples in the binary constant length case which contain an absolutely continuous component, and so we consider one on a four-letter alphabet and examine how this method applies and what exponents we obtain from it. 

Let $\varrho^{}_{\text{RS}}: a \mapsto ab, b \mapsto a\bar{b}, \bar{b} \mapsto \bar{a}b, \bar{a}\mapsto \bar{a}\bar{b}$ be the Rudin--Shapiro substitution, see \cite{BG:Pair} whose Fourier matrix is given by 
$$B(k)=\begin{pmatrix}
1 & 1 & 0 & 0\\
e^{2\pi ik} & 0 & e^{2\pi ik} & 0\\
0 & e^{2\pi ik} & 0 & e^{2\pi ik}\\
0 & 0 & 1 & 1
\end{pmatrix}.$$

Since $r=0$ is an eigenvalue of $B(k)$, the matrix fails to be invertible for any $k$, and hence we cannot directly define a cocyle corresponding to the outward iteration from $B$. Note however that there exists a $k-$independent matrix $S$ that transforms $B(k)$ into  block diagonal form, compare with \cite{BG:Pair}), which is explicitly given by  
\[
B^{\prime}(k)=
\left(
\begin{array}{c|c}
\begin{matrix}
1+e^{2\pi ik} & 0 \\
-e^{2\pi ik} & 0
\end{matrix} & \boldsymbol{O} \\
\hline
\boldsymbol{O} & \begin{matrix}
1 & 1 \\
e^{2\pi ik} & -e^{2\pi ik}
\end{matrix}
\end{array}
\right),
\]
with the matrix $C(k)$ in the upper block getting the zero eigenvalue and the lower block matrix $D(k)$ now being invertible for all $k$ $(\det D(k)=-2e^{2\pi ik})$. It is also worth noting that this decomposition into invariant subspaces is induced by the  
bar swap symmetry which the substitution satisfies, which is given by the identification $a\longleftrightarrow \bar{a},b\longleftrightarrow \bar{b}$, see \cite{BG:Pair}. The subspaces $V_C,V_D$ on which these block matrices act are known to carry the specific spectral components; see \cite{Bartlett,Queffelec}; i.e., that the pure point part lives in $V_C$ and the absolutely continuous component is in $V_D$. 
Since $D^{-1}(k)$ exists for all $k$, we can consider the cocyle defined by iterating this matrix and compute its corresponding Lyapunov exponent. From the irreducibility of $D$, it follows that the algebra it generates is the full algebra $M_{2}\left(\C\right)$, and that our technique of computing the exponent for each invariant subspace will not work. 
However, one notices immediately that $\frac{1}{\sqrt{2}}D\left(k\right)$ is unitary and so, for any starting vector, $v\left(k\right)$ ends up having the same norm as its $n-$th iterate $v\left(\lambda^{n}k\right)$ under the cocyle defined by $D$, which means that both Lyapunov exponents vanish, $\chi^{}_{1}=\chi^{}_{2}=0$. One can also check that this unitarity condition is satisfied by the restricted Fourier matrix of the substitution in \cite{ACspec}, which, by construction, is known to contain absolutely continuous diffraction. 
\exend
\end{example}
\section{Conclusion}
We have presented a method that confirms that the diffraction spectrum of such a binary substitution system does not contain an absolutely continuous component, which agrees with previous results in \cite{Bartlett,Queffelec}  predicting a 
singular continuous spectrum for the binary bijective case
or pure point as per Dekking in \cite{Dekking} for those with coincidences.
 
As mentioned in Remark \ref{nonpisot}, the method discussed here does not require the inflation multiplier to be an integer, and hence the same analysis could be used for more general substitutions, even those having non-Pisot inflation multipliers. However, when one foregoes the integer multiplier condition, one also does not get Oseledec's theorem automatically and so one must remedy this to prove the existence of the exponents in a set of full measure via some version of Kingman's subadditive ergodic theorem; see \cite{Schmeling}. 

Obviously, this can be extended to the general $n-$letter constant-length case, in which one can show positivity for all exponents when the substitution is bijective Abelian \cite{WIP2}. This together with Dekking's theorem recovers Bartlett's result in full. The non--Abelian case remains an open question.

It also gets computationally more involved when one deals with matrices that generate irreducible algebras because there are no subspaces where one can calculate the exponents directly. The unitarity of the restricted Fourier matrix in the Rudin--Shapiro case obviously does not hold in general; consider for example the substitution in \cite{CG17} which is the same example used in \cite{SSCL} to show that the $\sqrt{q}-$condition is necessary but not sufficient.  
In fact, one can numerically compute that the outward exponents for the a.e. invertible part of the Fourier matrix of this substitution are both positive.

\section{Acknowledgment}
The author is grateful to Michael Baake and Franz G\"ahler for their cooperation, and to Lax Chan and Uwe Grimm for discussions. He also would like to thank an anonymous referee for useful comments that helped improve this manuscript. This  work was financially supported by the German Research Foundation (DFG), within the CRC 701.


\end{document}